\documentclass[12pt,reqno]{amsart}
\usepackage{amsmath,a4wide}
\usepackage{xfrac}
\usepackage{stmaryrd,mathrsfs,bm,amsthm,mathtools,yfonts,amssymb,color}
\usepackage{xcolor}
\usepackage{epsfig,color,graphicx,graphics}
\usepackage{hyperref}

\begingroup
\newtheorem{theorem}{Theorem}[section]
\newtheorem{lemma}[theorem]{Lemma}

\newtheorem{corollary}[theorem]{Corollary}

\newtheorem{question}[theorem]{Question}
\endgroup

\theoremstyle{definition}
\newtheorem{definition}[theorem]{Definition}

\numberwithin{equation}{section}

\newcommand\supp{{\rm spt}}
\newcommand\res{\mathop{\hbox{\vrule height 7pt width .3pt depth 0pt
\vrule height .3pt width 5pt depth 0pt}}\nolimits}
\newcommand{\cH}{{\mathcal{H}}}
\def\a#1{\left\llbracket{#1}\right\rrbracket}
\newcommand\breg{{\rm Reg_b}}
\newcommand\bsing{{\rm Sing_b}}

\newcommand\N{{\mathbb N}}
\newcommand\R{{\mathbb R}}

\newcommand{\bB}{{\mathbf B}}

\newcommand{\qhalf}{\left(Q-{\textstyle \frac12}\right)}

\DeclareMathAlphabet{\mathpzc}{OT1}{pzc}{m}{it}

\newcommand{\gammaup}{\Gamma}
\newcommand{\gammado}{\gamma}

\def\XXint#1#2#3{{\setbox0=\hbox{$#1{#2#3}{\int}$ }
\vcenter{\hbox{$#2#3$ }}\kern-.6\wd0}}

\title[Boundary regularity of mass minimizers]{Boundary regularity of mass-minimizing integral currents and a question of Almgren}

\author{Camillo De Lellis}
\address{Institut f\"ur Mathematik, Universit\"at Z\"urich CH-8057 Z\"urich, Switzerland}
\email{camillo.delellis@math.uzh.ch}

\author{Guido de Philippis}
\address{SISSA Via Bonomea 265, I-34136 Trieste, Italy}
\email{gdephili@sissa.it}

\author{Jonas Hirsch}
\address{SISSA Via Bonomea 265, I-34136 Trieste, Italy}
\email{jonas.hirsch@sissa.it}

\author{Annalisa Massaccesi}
\address{Dipartimento di Informatica, Universit\`a di Verona Strada le Grazie 15, I-37134 Verona, Italy}
\email{annalisa.massaccesi@math.uzh.ch}

\begin{document}

\maketitle

\begin{abstract}
This short note is the announcement of a forthcoming work in which we prove a first general boundary regularity result for area-minimizing currents in higher codimension, without any geometric assumption on the boundary, except that it is an embedded submanifold of a Riemannian manifold, with a mild amount of smoothness ($C^{3, a_0}$ for a positive $a_0$ suffices). Our theorem allows to answer a question posed by Almgren at the end of his Big Regularity Paper. In this note we discuss the ideas of the proof and we also announce a theorem which shows that the boundary regularity is in general weaker that the interior regularity.
Moreover we remark an interesting elementary byproduct on boundary monotonicity formulae. 
\end{abstract}

\section{Introduction}

Consider a smooth complete Riemannian manifold $\Sigma$ of dimension $m+\bar n$ and a smooth closed submanifold $\gammaup \subset \Sigma$ of dimension $m-1$ which is a boundary in integral homology. Since the pioneering work of Federer and Fleming (cf. \cite{FF}) we know that $\gammaup$ bounds an integer rectifiable current $T$ in $\Sigma$ which minimizes the mass among all integer rectifiable currents bounded by $\gammaup$. 

In general, consider an open $U\subset \Sigma$ and a submanifold $\gammaup\subset \Sigma$ which has no boundary in $U$. If $T$ is an integral current in $U$ with $\partial T \res U= \a{\gammaup}\res  U$ we say that $T$ is mass-minimizing if 
\[
\mathbf{M} (T+\partial S) \geq \mathbf{M} (T)
\]
for every integral current $S$ in $U$. 

Starting with the pioneering work of De Giorgi (see \cite{DG}) and thanks to the efforts of several mathematicians in 
the sixties and the seventies (see \cite{Fleming, DeGiorgi5,Alm2,Simons}), it is known that, if $\Sigma$ is of class $C^{2,a_0}$ for some $a_0>0$, in codimension $1$ (i.e., when $\bar n =1$) and away from the boundary $\gammaup$, $T$ is a smooth submanifold except for a relatively closed set of Hausdorff dimension at most $m-7$. Such set, which from now on we will call {\em interior singular set}, is indeed $(m-7)$-rectifiable (cf. \cite{Simon2}) and it has been recently proved that it must have locally finite Hausdorff $(m-7)$-dimensional measure (see \cite{NV}). 
In higher codimension, namely when $\bar n =2$, Almgren proved in a monumental work (known as Almgren's Big Regularity Paper \cite{Alm}) that, if $\Sigma$ is of class $C^5$, then the interior singular set of $T$ has Hausdorff dimension at most $m-2$. In a series of papers (cf. \cite{DS1,DS2,DS3,DS4,DS5}) the first author and Emanuele Spadaro have revisited Almgren's theory introducing several new ideas which simplify his proof considerably. Furthermore, the first author together with Spadaro and Spolaor, in \cite{DSS1,DSS2,DSS3,DSS4} applied these sets of ideas to establish a complete proof of Chang's interior regularity results for \(2\) dimensional mass-minimizing 
currents \cite{Chang}, showing that in this case interior  singular points are isolated. 

\medskip

Both in codimension one and in higher codimension the interior regularity theory described above is, in terms of dimensional bounds for the singular set, optimal (cf. \cite{BDG} and \cite{Fed}). In the case of boundary points the situation is instead much less satisfactory. The first boundary regularity result is due to Allard who, in his Ph.D. thesis (cf. \cite{AllPhD}), proved that, if $\Sigma = \mathbb R^{m+\bar n}$ and $\gammaup$ is lying on the boundary of a uniformly convex set, then for every point $p\in \gammaup$ there is a neighborhood $W$ such that $T\res W$ is a classical oriented submanifold (counted with multiplicity $1$) whose boundary (in the usual sense of differential topology) is $\gammaup \cap W$. In his later paper \cite{AllB} Allard developed a more general boundary regularity theory from which he concluded the above result as a simpler corollary. 

When we drop the ``convexity assumption'' described above, the same conclusion cannot be reached.
Let for instance $\gammaup$ be the union of two concentric circles $\gamma_1$ and $\gamma_2$ which are contained in a given $2$-dimensional plane $\pi_0\subset \mathbb R^{2+\bar n}$ and have the same orientation. Then the area-minimizing current $T$ in $\mathbb R^{2+\bar n}$ which bounds $\gammaup$ is unique and it is the sum of the two disks bounded by $\gamma_1$ and $\gamma_2$ in $\pi_0$, respectively. At every point $p$ which belongs to the inner circle the current $T$ is ``passing'' through the circle while the multiplicity jumps from $2$ to $1$.
However it is natural to consider such points as ``regular'', motivating therefore the following definition.

\begin{definition}\label{def:reg_points}
A point $x\in\gammaup$ is a regular point for $T$ if there exist a neighborhood $W\ni x$ and a regular $m$-dimensional connected submanifold $\Sigma_0\subset W\cap\Sigma$ (without boundary in $W$) such that $\supp (T)\cap W\subset \Sigma_0$. The set of such points will be denoted by $\breg(T)$ and its complement in $\gammaup$ will be denoted by $\bsing(T)$.
\end{definition}

By the Constancy Lemma, if $x\in \gammaup$ is a regular point, if $\Sigma_0$ is as in Definition \ref{def:reg_points} and if the neighborhood $W$ is sufficiently small, then the following holds:
\begin{itemize}
\item $\gammaup\cap W$ is necessarily contained in $\Sigma_0$ and divides it in two disjoint regular submanifolds $\Sigma_0^+$ and $\Sigma_0^-$ of $W$ with boundaries $\pm\gammaup$;
\item there is a positive 
$Q\in\N$  such that $T\res V=Q\a{\Sigma_0^+}+(Q-1)\a{\Sigma_0^-}$. 
\end{itemize}
We define the density of such points $p$ as $Q-\frac 12$ and we denote it by $\Theta(T,p)=Q-\frac 12$. 

If the density is $\frac{1}{2}$ then the point fulfills the conclusions of Allard's boundary regularity theorem and $\Sigma_0$ is not uniquely determined: the interesting geometrical object is $\Sigma_0^+$ and any smooth ``extension'' of it across $\gammaup$ can be taken as $\Sigma_0$. On the other hand for $Q\geq 2$ the local behavior of the current is similar
to the example of the two circles above: it is easy to see that $\Sigma_0$ is uniquely determined and that it has mean curvature zero.

\medskip
When the codimension of the area-minimizing current is $1$, Hardt and Simon proved in \cite{HS} that the set of boundary singular points is empty, hence solving completely the boundary regularity problem when $\bar n =1$ (although the paper \cite{HS} deals only with the case $\Sigma = \mathbb R^{m+\bar n}$, its extension to a general Riemannian ambient manifold should not cause real issues). In the case of general codimension and general $\gammaup$, Allard's theory implies the existence of (relatively few) boundary regular points only in special ambient manifolds $\Sigma$: for instance when $\Sigma = \mathbb R^{m+\bar n}$ we can recover the regularity of the ``outermost'' boundary points $q\in \gammaup$
(i.e., those points $q$ where $\gammaup$ touches the smallest closed ball which contains it). According to the existing literature, however, we cannot even exclude that the set of regular points is empty when $\Sigma$ is a closed Riemannian manifold. 
In the last remark of the last section of his Big Regularity Paper, cf. \cite[Section 5.23, p. 835]{Alm}, Almgren states the following open problem,
which is closely related to the discussion carried above.

\begin{question}[Almgren]\label{q:alm}
``I do not know if it is possible that the set of density $\frac{1}{2}$ points is empty when $U= \Sigma$ and $\gammaup$ is connected.''
\end{question}

The interest of Almgren in Question \ref{q:alm} is motivated by an important geometric conclusion: in \cite[Section 5.23]{Alm} he shows that, if there is at least one density $\frac{1}{2}$ point and $\gammaup$ is connected, then 
$\supp (T)$ is as well connected and the current $T$ has (therefore) multiplicity $1$ almost everywhere. In other words
the mass of $T$ coincides with the Hausdorff $m$-dimensional measure of its interior regular set. 

\medskip

In the forthcoming paper \cite{DDHM} we show the first general boundary regularity result in any codimension, which guarantees the density of boundary regular points without any restriction (except for a mild regularity assumption on $\gammaup$ and $\Sigma$: both are assumed to be of class $C^{3,a_0}$ for some positive $a_0$). As a corollary we answer Almgren's question in full generality showing: when $U=\Sigma$ and $\gammaup$ is connected, then there is always at least
one point of density $\frac{1}{2}$ and the support of any minimizer is connected. In the next section we will state the main results of \cite{DDHM}, whereas in Section \ref{s:sketch} we will give an account of their (quite long) proofs. Finally, in Section \ref{s:monotonia} we outline an interesting side remark sparked by one of the key computations in \cite{DDHM}. The latter
yields an alternative proof of Allard's boundary monotonicity formula under slightly different assumptions: in particular it covers, at the same time, the Gr\"uter-Jost monotonicity formula for free boundary stationary varifolds. 

\section{Main theorems} \label{s:main}

Our main result in \cite{DDHM} is the following

\begin{theorem}\label{thm:main}
Consider a $C^{3,a_0}$ complete Riemannian submanifold $\Sigma\subset\R^{m+n}$ of dimension $m+\overline n$ and an open set $W\subset \mathbb R^{m+n}$. Let $\gammaup\subset\Sigma\cap W$ be a $C^{3,a_0}$ oriented submanifold without boundary in $W\cap \Sigma$ and let $T$ be an integral $m$-dimensional mass-minimizing current in $W \cap \Sigma$ with boundary $\partial T \res W =\a{\gammaup}$.
Then $\breg(T)$ is dense in $\gammaup$.
\end{theorem}

As a simple corollary of the theorem above, we conclude that Almgren's Question \ref{q:alm} has a positive answer. 

\begin{corollary}
Let $W= \mathbb R^{m+n}$ and assume $\Sigma, \gammaup$ and $T$ are as in Theorem \ref{thm:main}. If $\gammaup$ is connected, then
\begin{itemize}
\item Every point in $\breg(T)$ has density $\frac{1}{2}$;
\item The support $\supp (T)$ of the current $T$ is connected;
\item The multiplicity of the current is $1$ at $\mathcal{H}^m$-a.e. interior point, and so the mass of the current coincides with $\mathcal{H}^m (\supp (T))$.
\end{itemize}
\end{corollary}

In fact the above corollary is just a case of a more general ``structural'' result, which is also a consequence of Theorem
\ref{thm:main}.

\begin{theorem}\label{thm:primes}
Let $W= \mathbb R^{m+n}$ and assume $\Sigma, \gammaup$ and $T$ are as in Theorem \ref{thm:main}. Denote by $\Gamma_1,\ldots,\Gamma_N$ the connected components of $\gammaup$. Then
\begin{equation}\label{e:dec_indec}
T=\sum_{j=1}^{\overline N} Q_jT_j\,, 
\end{equation}
where:
\begin{itemize}
\item[(a)]  For every $j=1,\ldots,\overline N$, $T_j$ is an integral current with $\partial T_j=\sum_{i=1}^{N}\sigma_{ij}\a{ \Gamma_i}$ and $\sigma_{ij}\in\{-1,0,1\}$.
\item[(b)] For every $j=1,\ldots,\overline N$, $T_j$ is an area-minimizing current and $T_j=\cH^{m}\res \Lambda_j$, where $\Lambda_1,\ldots,\Lambda_{\overline N}$ are the connected components of $\supp(T)\setminus\Gamma$.
\item[(c)] Each $\Gamma_i$ is 
\begin{itemize} 
\item either ``one-sided'', which means that all coefficients $\sigma_{ij}=0$ except for one $j=o(i)$ for which $\sigma_{io(i)}=1$;
\item or ``two-sided'', which means that:
\begin{itemize}
\item there is one $j=p(i)$ such that $\sigma_{ip(i)}=1$,
\item there is one $j=n(i)$ such that $\sigma_{in(i)}=-1$,
\item all other $\sigma_{ij}=0$.
\end{itemize}
\end{itemize}
\item[(d)] If $\Gamma_i$ is one-sided, then $Q_{o(i)}=1$ and all points in $\Gamma_i\cap \breg{T}$ have multiplicity $\frac 12$. 
\item[(e)] If $\Gamma_i$ is two-sided, then $Q_{n(i)}=Q_{p(i)}-1$, all points in $\Gamma_i\cap\breg{T}$ have multiplicity $Q_{p(i)}-\frac 12$ and $T_{p(i)}+T_{n(i)}$ is area minimizing.
\end{itemize}
\end{theorem}

Note that, as a simple consequence of Theorem \ref{thm:primes} and the interior regularity theory, we conclude that in every ``two-sided'' component $\Gamma_i$ of the boundary $\gammaup$ the boundary singular points have dimension at most $m-2$.

In view of the interior regularity results, one might be tempted to conjecture that Theorem \ref{thm:main} is very suboptimal and that the Hausdorff dimension of ${\rm Sing_b} (T)$ is at most \(m-2\). Though currently we do not have an answer to this question, let us stress that at the boundary some new phenomena arise. Indeed, in \cite{DDHM}, we can prove the following:

\begin{theorem}\label{thm:example}
There are a smooth closed simple curve $\Gamma \subset \mathbb R^4$ and a mass minimizing current $T$ in $\mathbb R^4$ such that $\partial T = \a{\Gamma}$ and $\bsing(T)$ has an accumulation point.
\end{theorem}

In particular Chang's result, namely the discreteness of {\em interior} singular points for two dimensional mass minimizing currents, does not hold at the boundary. Actually the example can be modified in order to obtain also a sequence of interior singular points accumulating towards the boundary, see \cite{DDHM}.

\section{The main steps to Theorem \ref{thm:main}}\label{s:sketch}

In this section we outline the long road which is taken in \cite{DDHM} to prove Theorem
\ref{thm:main}. We fix therefore $\Sigma, \gammaup$ and $T$ as in Theorem \ref{thm:main}.

\subsection{Reduction to collapsed points} Recalling Allard's monotonicity formula, we introduce at each boundary point $p\in \gammaup$ the density $\Theta (T, p)$, namely the limit, as $r\downarrow 0$ of the normalized mass ratio in the ball $\bB_r (p)\subset \mathbb R^{m+n}$ (in particular the normalization is chosen so that at regular boundary points the density coincides with the one defined in the previous section).  Using a suitable variant of Almgren's stratification theorem, we conclude first that, except for a set of Hausdorff dimension at most $m-2$, at any boundary point $p$ there is a tangent cone which is ``flat'', namely which is contained in an $m$-dimensional plane $\pi \supset T_0 \gammaup$. Secondly, using a classical upper Baire category argument, we show that, for a dense subset of boundary points $p$, additionally to the existence of a flat tangent cone, there is a sufficiently small neighborhood $U$ where the density $\Theta (T,q)$ is bounded below, at any $q\in \Gamma \cap U$, by $\Theta (T, p)$. In particular 
the proof of Theorem \ref{thm:main} is reduced to the claim that any such point, which we call ``collapsed'', is in fact regular.

\subsection{The ``linear'' theory} Assume next that $0\in \gammaup$ is a collapsed point and let $Q-\frac{1}{2}$ be its density. By Allard's boundary regularity theory for stationary varifolds, we know a priori that $0$ is a regular point if $Q=1$ and thus we can assume, without loss of generality, that $Q \geq 2$. Fix a flat tangent cone to $0$ and assume, up to rotations, that it is the plane $\pi_0 = \mathbb R^m \times \{0\}$ and that $T_0 \gammaup = \{x_m=0\} \cap \pi_0$. Denote by $\pi_0^\pm$ the two half-planes $\pi_0^\pm = \{\pm x_m > 0\} \cap \pi_0$. Assume for the moment that, at suitably chosen small scales, the current $T$ is formed by $Q$ sheets over $\pi_0^+$ and $Q-1$ sheets over $\pi_0^-$. By a simple linearization argument such sheets must then be almost harmonic (in a suitable sense). 

Having this picture in mind, it is natural to develop a theory of $\qhalf$-valued functions minimizing the Dirichlet energy. In order to explain the latter object consider the projection $\gammado$ of $\gammaup$ onto $\pi_0$. On a sufficiently small disk $\bB_r (0) \cap \pi_0$, $\gammado$ divides $\pi_0$ into two regions. A Lipschitz $\qhalf$-valued map consists of:
\begin{itemize}
\item a Lipschitz $Q$-valued map (in the sense of Almgren, cf. \cite{DS1}) $u^+$ on one side of $\gammado$
\item and a Lipschitz $(Q-1)$-valued map $u^-$ on the other side,
\end{itemize}
satisfying the compatibility condition that the union of their graphs form a current whose boundary is the submanifold $\gammaup$ itself. A $\qhalf$-map will then be called ${\rm Dir}$-minimizing if it minimizes the sum of the Dirichlet energies of the two ``portions'' $u^+$ and $u^-$ under the constraint that $\Gamma$ and the boundary values on $\partial (\bB_r (0) \cap \pi_0)$ are both fixed. 

The right counterpart of the ``collapsed point'' situation described above is the assumption that all the $2Q-1$ sheets meet at their common boundary $\gammaup$; under such assumption we say that the $\qhalf$ ${\rm Dir}$-minimizer has collapsed interface. We then develop a suitable regularity theory for minimizers with collapsed interface. First of all their H\"older continuity follows directly from the Ph.D. thesis of the third author, cf. \cite{Jonas}. Secondly, the most important conclusion of our analysis is that a minimizer can have collapsed interface only if it consists of a single harmonic sheet ``passing through'' the boundary data, counted therefore with multiplicity $Q$ on one side and with multiplicity $Q-1$ on the other side. 

The latter theorem is ultimately the ``deus ex machina'' of the entire argument leading
to Theorem \ref{thm:main}. The underlying reason for its validity is that a monotonicity formula for a suitable variant of Almgren's frequency function holds. Given the discussion of \cite{Jonas2}, such monotonicity can only be hoped in the collapsed situation and, remarkably, this suffices to carry on our program. 

The validity of the monotonicity formula is clear when the collapsed interface is flat.
However, when we have a curved boundary, a subtle yet important point becomes crucial: we cannot hope in general for the exact first variation identities which led Almgren to his monotonicity formula, but we must replace them with suitable inequalities. Moreover the latter can be achieved only if we adapt the frequency function by integrating a suitable weight. We illustrate this idea in a simpler setting in the next section. 

\subsection{First Lipschitz approximation} A first use of the linear theory is approximating the current with the graph of a Lipschitz $\qhalf$-valued map around collapsed points. The approximation is then shown to be almost ${\rm Dir}$-minimizing. Our approximation algorithm is a suitable adaptation of the one developed in \cite{DS3} for interior points. In particular, after adding an ``artificial sheet'', we can directly use the Jerrard-Soner modified BV estimates of \cite{DS3} to give a rather accurate Lipschitz approximation: the subtle point is to engineer the approximation so that it has collapsed interface. 

\subsection{Height bound and excess decay} The previous Lipschitz approximation, together with the linear regularity theory, is used to establish a power-law decay of the excess \`a la De Giorgi in a neighborhood of a collapsed point. The effect of such theorem is that the tangent cone is flat and unique at every point $p\in \gammaup$ in a sufficiently small neighborhood of the collapsed point $0\in \gammaup$. Correspondingly, the plane $\pi (p)$ which contains such tangent cone is H\"older continuous in the variable $p\in \gammaup$ and the current is contained in a suitable horned neighborhood of the union of such $\pi (p)$. 

An essential ingredient of our argument is an accurate height bound in a neighborhood of any collapsed point in terms of the spherical excess. The argument follows an important idea of Hardt and Simon in \cite{HS} and takes advantage of an appropriate 
variant of Moser's iteration on varifolds, due to Allard, combined with a crucial use of the remainder in the monotonicity formula. The same argument has been also used by Spolaor in a similar context in \cite{Spolaor}, where he combines it with the decay of the energy for ${\rm Dir}$-minimizers, cf. \cite[Proposition 5.1 \& Lemma 5.2]{Spolaor}.

\subsection{Second Lipschitz approximation} The decay of the excess proved in the previous step is used then to improve the accuracy of the Lipschitz approximation. In particular, by suitably decomposing the domain of the approximating map in a Whitney-type cubical decomposition which refines towards the boundary, we can take advantage of the interior approximation theorem of \cite{DS3} on each cube and then patch the corresponding graphs together. 

\subsection{Left and right center manifolds} The previous approximation result is combined with a careful smoothing and patching argument to construct a ``left'' and a ``right'' center manifold $\mathcal{M}^+$ and $\mathcal{M}^-$. The $\mathcal{M}^\pm$ are $C^{3,\kappa}$ submanifolds of $\Sigma$ with boundary $\gammaup$ and they provide a good approximation of the ``average of the sheets'' on both sides of $\gammaup$ in a neighborhood of the collapsed point $0\in \gammaup$. They can be glued together to form a $C^{1,1}$ submanifold $\mathcal{M}$ which ``passes through $\gammaup$'': each portion has $C^{3,\kappa}$ estimates {\em up to the boundary}, but we only know that the tangent spaces at the boundary coincide: we have a priori no information on the higher derivatives. The construction algorithm follows closely that of \cite{DS4} for the interior, but some estimates must be carefully adapted in order to ensure the needed boundary regularity. 

The center manifolds are coupled with two suitable approximating maps $N^\pm$. The latter take values on the normal bundles of $\mathcal{M}^\pm$ and provide an accurate approximation of the current $T$. Their construction is a minor variant of the one in \cite{DS4}. 

\subsection{Monotonicity of the frequency function and final blow-up argument} After constructing the center manifolds and the corresponding approximations we use a suitable Taylor expansion of the area functional to show that the monotonicity of the frequency function holds for the approximating maps $N^\pm$ as well. 

We then complete the proof of Theorem \ref{thm:main}: in particular we show that, if $0$ were a singular collapsed point, suitable rescalings of the approximating maps $N^\pm$ would produce, in the limit, a $\qhalf$ ${\rm Dir}$-minimizer violating the linear regularity theory. On the one hand the estimate on the frequency function plays a primary role in showing that the limiting map is nontrivial. On the other hand the properties of the center manifolds $\mathcal{M}^\pm$ enter in a fundamental way in showing that the average of the sheets of the limiting $\qhalf$ map is zero on both sides.  

\section{Weighted monotonicity formulae}\label{s:monotonia}

In this section we want to illustrate in a simple situation an idea which, in spite of being elementary, plays a fundamental role in our proof of Theorem \ref{thm:main}: boundary monotonicity formulae can be derived from the arguments of their interior counterparts provided we introduce a suitable weight. 

Let $\Sigma$ be an $(m-1)$-dimensional submanifold of $\mathbb R^{m+n}$. Consider an $m$-dimensional varifold $V$ in $\mathbb R^{m+n}\setminus \Gamma$ and assume it is stationary in $\mathbb R^{m+n} \setminus \Gamma$. Allard in \cite{AllB} derived his famous monotonicity formula at the boundary under the additional assumption that the density of $V$ has a uniform positive lower bound. His proof consists of two steps: he first derives a suitable representation for the first variation $\delta V$ of $V$ along general vector fields of $\mathbb R^{m+n}$, i.e., vector fields which might be nonzero on $\Gamma$. He then follows the derivation of the interior monotonicity formula, i.e., he tests the first variation along suitable radial vector fields. His proof needs the lower density assumption in the first part and although the latter can be removed (cf. \cite{AllPer}), the resulting argument is rather laborious. 

We introduce here varifolds which are stationary along ``tangent fields'':

\begin{definition}\label{d:stationary}
Consider an $m$-dimensional varifold $V$ in an open set $U \subset \mathbb R^{m+n}$ and let $\Gamma$ be a $k$-dimensional $C^1$ submanifold of $U$. We say that $V$ is stationary with respect to vector fields tangent to $\Gamma$ if
\begin{equation}
\delta V (\chi) =0 \qquad \mbox{for all $\chi\in C^1_c (U, \mathbb R^{m+n})$ which are tangent to $\Gamma$.}
\end{equation} 
\end{definition}

Clearly, when $k=m-1$, the condition above is stronger than that used by Allard in \cite{AllB}, where $\chi$ is assumed to {\em vanish} on $\Gamma$. On the other hand our condition is the natural one satisfied by classical minimal surfaces with boundary $\Gamma$, since the one-parameter family of isotopies generated by $\chi$ maps $\Gamma$ onto itself. When $k > m-1$, the condition is the one satisfied by classical ``free-boundary'' minimal surfaces, namely minimal surfaces with boundary contained in $\Gamma$ and meeting it orthogonally. In the context of varifolds, the latter have been considered by Gr\"uter and Jost in \cite{GJ}, where the two authors derived also an analog of Allard's monotonicity formula. In this section we show how one can take advantage of a suitable distortion of the Euclidean balls to give a (rather elementary) unified  approach to monotonicity formulae in both contexts.

\begin{definition}\label{d:distance}
Assume that $0\in \Gamma$.
We say that the function $d: \mathbb R^{m+n}\to \R$ is a distortion of the distance function adapted to $\Gamma$ if the following two conditions hold:
\begin{itemize}
\item[(a)] $d$ is of class $C^2$ on $\mathbb R^{m+n}\setminus \{0\}$ and $D^j d (x) = D^j |x| + O (|x|^{1-j+\alpha})$ for some fixed $\alpha\in (0,1]$ and for $j=0,1,2$;
\item[(b)] $\nabla d$ is tangent to $\Gamma$. 
\end{itemize}
\end{definition}

The following lemma is a simple consequence of the Tubular Neighborhood Theorem and it is left to the reader.

\begin{lemma}
If $\Gamma$ is of class $C^3$ then there is a distortion of the distance function adapted to $\Gamma$ where the exponent $\alpha$ of Definition \ref{d:distance}(a) can be taken to be $1$. 
\end{lemma}


The main point of our discussion is then the argument given below for the following

\begin{theorem}\label{t:monot}
Consider $\Gamma$ and $V$ as in Definition \ref{d:stationary}, assume that $0\in \Gamma$ and that $d$ is a distorted distance function adapted to $\Gamma$. Let $\varphi \in C^1_c ([0,1))$ be a nonincreasing function which is constant in a neighborhood of the origin. If $\alpha$ is the exponent of Definition \ref{d:distance}(a), then there are positive constants $C$ and $\rho$ such that the following inequality holds  for every positive $s< \rho$
\begin{equation}\label{e:differenziale}
\frac{d}{ds} \left[e^{Cs^\alpha} s^{-m}\int \varphi \left({\textstyle{\frac{d(x)}{s}}}\right)\, d\|V\| (x)\right]
\geq - e^{Cs^\alpha} s^{-m} \int \varphi ' \left({\textstyle{\frac{d(x)}{s}}}\right) \left| P_{\pi^\perp} \left(\frac{\nabla d (x)}{|\nabla d (x)|}\right)\right|^2\, dV (x, \pi)\, 
\end{equation} 
(where $P_\tau$ denotes the orthogonal projection on the subspace $\tau$). 
\end{theorem}

Note that if we let $\varphi$ converge to the indicator function of the interval $[0,1)$ we easily conclude that
\[
s\mapsto \Phi (s):= e^{Cs^\alpha} \frac{\|V\| (\{d<s\})}{s^m}\, 
\]
is monotone nondecreasing: indeed, for $\rho> s> r >0$, the difference $\Phi (s)-\Phi (r)$ controls the integral of a suitable nonnegative expression involving $d$ and the projection of $\nabla d/|\nabla d|$ over $\pi^\perp$. When $d(x)=|x|$, namely when $\Gamma$ is flat, the exponential weight disappears (i.e., the constant $C$ might be taken to be $0$), the inequality becomes an equality and (in the limit of $\varphi\uparrow \mathbf{1}_{[0,1)}$) we recover Allard's identity
\[
\frac{\|V\| (\bB_s (0))}{\omega_m s^m} - \frac{\|V\| (\bB_r (0))}{\omega_m r^m} = \int_{\bB_s (0) \setminus \bB_r (0)} \frac{|P_{\pi^\perp} (x)|^2}{|x|^{m+2}} d\|V\| (x)\, .
\] 
In particular, since $d$ is asymptotic to $|x|$, all the conclusions which are usually derived from Allard's theorem (existence of the density and its upper semicontinuity, conicity of the tangent varifolds, Federer's reduction argument and Almgren's stratification) can be derived from Theorem \ref{t:monot} as well. Moreover, the argument given below can be easily extended 
to cover the more general situation of varifolds with mean curvature satisfying a suitable integrability condition. 

\begin{proof}[Proof of Theorem \ref{t:monot}] Consider the vector field
\[
X_s (x) = \varphi \left(\frac{d(x)}{s}\right) d(x) \frac{\nabla d(x)}{|\nabla d(x)|^2}\, .
\]
$X_s$ is obviously $C^1$ on $\mathbb R^{m+n}\setminus \{0\}$ and moreover we have
\begin{align*}
DX_s &=  \varphi \left(\frac{d(x)}{s}\right) \left[
\frac{\nabla d \otimes \nabla d}{|\nabla d|^2} + \frac{d D^2 d}{|\nabla d|^2} - 2d \frac{\nabla d}{|\nabla d|^4} \otimes (D^2 d \cdot \nabla d)\right]
 + \varphi' \left(\frac{d(x)}{s}\right) \frac{d}{s} \frac{\nabla d \otimes \nabla d}{|\nabla d|^2}\, .
\end{align*}
From the above formula, using that $\varphi$ is constant in a neighborhood of the origin and Definition \ref{d:distance}(a), we easily infer 
\[
DX_s (x) = \varphi \left(\frac{d(x)}{s}\right){\rm Id}\, + O (|x|^\alpha). 
\]
In particular $X_s$ is $C^1$, compactly supported in $U$ (provided $s$ is sufficiently small), and tangent to $\Gamma$. Thus 
\[
0 = \delta V (\chi) = \int {\rm div}_\pi\, X_s (p)\, dV (p, \pi) = 0\, .
\] 
Fix next an orthonormal basis $e_1, \ldots , e_m$ of $\pi$ and use Definition \ref{d:distance}(a) to compute
\begin{align*}
{\rm div}_\pi\, X_s &= \sum_{i=1}^m e_i^T \cdot DX \cdot e_i = (m + O (s^\alpha))\varphi \left(\frac{d(x)}{s}\right)+\varphi' \left(\frac{d(x)}{s}\right) \sum_i \frac{|\nabla d \cdot e_i|^2}{|\nabla d|^2}\\
&=(m + O (s^\alpha)) \varphi \left(\frac{d(x)}{s}\right) + \varphi '\left(\frac{d(x)}{s}\right) \left(1- \left|P_{\pi^\perp} \left(\frac{\nabla d}{|\nabla d|}\right)\right|^2 \right)\, .
\end{align*}
Plugging the latter identity in the first variation condition we achieve the following inequality for a sufficiently large constant $C$:
\begin{align*}
 &\int \left(-m  \varphi \left(\frac{d(x)}{s}\right) -\varphi' \left(\frac{d(x)}{s}\right) \frac{d (x)}{s}\right)\, d\|V\| (x) + C \alpha s^\alpha \int \varphi \left(\frac{d(x)}{s}\right) 
d\|V\| (x)\\
 \geq\; & - \int \varphi' \left(\frac{d(x)}{s}\right) \left| P_{\pi^\perp} \left(\frac{\nabla d (x)}{|\nabla d (x)|}\right)\right|^2\, dV (x, \pi)\, .
\end{align*}
Multiplying both sides of the inequality by $ e^{Cs^\alpha} s^{-m-1}$ we then conclude \eqref{e:differenziale}. 
\end{proof}

\bibliographystyle{plain}


\end{document}